\documentclass[12pt]{amsart}
\usepackage[margin=3cm]{geometry}

\newtheorem{theorem}{Theorem}[section]

\newtheorem{corollary}[theorem]{Corollary}
\newtheorem{lemma}[theorem]{Lemma}

\theoremstyle{definition}

\theoremstyle{remark} \theoremstyle{remark}

\numberwithin{equation}{section}

\newcommand{\C}{\mathbb{C}}

\newcounter{rom}
\renewcommand{\therom}{(\roman{rom})}
\newenvironment{romanlist}{\begin{list}{\therom}
                {\setlength{\leftmargin}{2em}\usecounter{rom}}}%
{\end{list}}
\title{Null pseudo-isotropic Lagrangian surfaces}
\begin{document}
\author[A. Carriazo]{Alfonso Carriazo}
 \address{Departamento de Geometr\'{i}a y Topolog\'{i}a, Facultad de Matem\'aticas, c/ Tarfia s/n, Universidad de Sevilla, 41012, Sevilla, Spain}
 \email{carriazo@us.es}

\author[V. Mart\'{i}n-Molina]{Ver\'onica  Mart\'{i}n-Molina}
 \address{Departamento de Did\'{a}ctica de las Matem\'{a}ticas, Facultad de Ciencias de la Educaci\'on, c/ Pirotecnia s/n, Universidad de Sevilla, 41013, Sevilla, Spain}
 \email{veronicamartin@us.es}

\author[L. Vrancken]{Luc Vrancken}
 \address{LAMAV, Universit\'{e} de Valenciennes, 59313, Valenciennes Cedex 9, France \and Departement Wiskunde, KU Leuven, Celestijnenlaan 200 B, 3001, Leuven, Belgium}
 \email{luc.vrancken@univ-valenciennes.fr}

\thanks{The first two authors are partially supported by the MINECO-FEDER grant MTM2014-52197-P. They also belong to the PAIDI groups FQM-327 and FQM-226 (Junta de Andaluc\'ia, Spain), respectively.}

\date{}
\begin{abstract}\noindent
In this paper we will show that a Lagrangian, Lorentzian surface $M^2_1$ in a complex pseudo space form $\widetilde M^2_1 (4c)$
is pseudo-isotropic if and only if $M$ is minimal.
Next we will obtain a complete classification of all Lagrangian, Lorentzian surfaces which are lightlike pseudo-isotropic but not pseudo-isotropic.
\end{abstract}

\maketitle

\sloppy

{{\bfseries Key words}:  {\em Lagrangian submanifold, complex
projective space, isotropic submanifold, Lorentzian submanifold}.

{\bfseries Subject class: } 53B25, 53B20.}

\section{Introduction}
The notion of isotropic submanifold was first introduced in \cite{ON2}
by O'Neill for immersions of Riemannian manifolds and recently extended by Cabrerizo, Fern\'andez and G\'omez in \cite{CFG} to the pseudo-Riemannian case. A submanifold is called pseudo-isotropic if, for any point $p$ and any tangent vector $v$ at a point $p$, we have that
\begin{equation} \label{1.1}
\left<h(v,v),h(v,v)\right>= \widetilde\lambda(p) \left<v,v\right>^2,
\end{equation}
where $h$ denotes the second fundamental form of the immersion and $\widetilde\lambda$ is a smooth function on the submanifold.

Note that since the induced metric is pseudo-Riemannian it is natural to distinguish between timelike, spacelike and lightlike (or null) vectors. This leads in a natural way to the notions of
\begin{romanlist}
\item Timelike pseudo-isotropic
 if, for any point $p$ and any timelike tangent vector $v$ at a point $p$, equation \eqref{1.1} is satisfied,
\item Spacelike pseudo-isotropic
 if, for any point $p$ and any spacelike tangent vector $v$ at a point $p$, equation \eqref{1.1} is satisfied,
\item Lightlike isotropic if, for every lightlike vector $v$
at the point $p$, we have that $h(v,v)$ is again a lightlike vector.
\end{romanlist}

It was shown in \cite{CFG} that the notions of pseudo-isotropic, timelike pseudo-isotropic and spacelike pseudo-isotropic are equivalent. In the same paper they also included an example of an immersion which is lightlike pseudo-isotropic but not  pseudo-isotropic.

Here we are particularly interested in Lagrangian immersions of complex space forms. In the positive definite case, isotropic Lagrangian immersions have been studied in \cite{C1}, \cite{C2}, \cite{LW}, \cite{MU} and \cite{V}.
In this paper we will consider pseudo-isotropic and lightlike pseudo-isotropic Lagrangian, Lorentzian surfaces $M^2_1$ in a complex pseudo space form $\widetilde M^2_1 (4c)$. We will assume that the space form is not definite and hence has real signature 2. By changing the sign of the metric if necessary, it is sufficient to deal with the cases $c=0$ or $c=1$.

We will first show in Section \ref{sec-3}  that
\begin{theorem} \label{th1.1}
Let $M$ be a Lagrangian, Lorentzian surface in a complex space form. Then $M$ is pseudo-isotropic if and only if $M$ is minimal.
\end{theorem}

Next we will obtain a complete classification of all Lagrangian, Lorentzian surfaces which are lightlike pseudo-isotropic but not pseudo-isotropic. To do so, we will first show in Section \ref{sec-4} that such a surface can be seen as the union of two surfaces that we will call of Type I and of Type II. These will be classified, case by case, in Sections \ref{sec-5} and \ref{sec-6}, respectively.

\section{Preliminaries}
Throughout this paper we will assume that $M$ is a Lagrangian, Lorentzian submanifold $M$ of a complex space form $\widetilde M$. We use the standard formulas of Gauss and Weingarten for a submanifold, introducing the
second fundamental form $h$ and the shape operators $A$ by
\begin{align*}
\widetilde \nabla_X Y &= \nabla_X Y + h(X,Y),\\
\widetilde \nabla_X \xi &=-A_{\xi} X + \nabla_X^\perp \xi,
\end{align*}
where $X$ and $Y$ are tangent vector fields and $\xi$ is normal.
Here, as usual, $\widetilde \nabla$ denotes the Levi-Civita connection on the ambient space and, if no confusion is possible, we will always identify $M$ with its image in $\widetilde M$.

Since $M$ is Lagrangian, we have that the complex structure $J$ interchanges the tangent and the normal spaces. Using the formulas of Gauss and Codazzi this implies that
\begin{align*}
\nabla_X^\perp JY &= J \nabla_X Y,\\
A_{JX} Y &= -J h(X,Y)=A_{JY} X.
\end{align*}
The latter formula implies that the cubic form $<h(X,Y), JZ>$ is totally symmetric in all components.

We denote the curvature tensors of $\nabla$ and $\nabla^{\perp}$ by
$R$ and $R^{\perp}$, respectively. The first covariant
derivative of $h$ is defined by
\begin{equation}
\begin{aligned}
(\nabla h)(X,Y,Z)=\nabla^{\perp}_X h(Y,Z)-h(\nabla_X
Y,Z)-h(\nabla_X Z,Y),\label{2.1}
\end{aligned}
\end{equation}
where $X$, $Y$, $Z$ and $W$ are tangent vector fields.

The equations of Gauss, Codazzi and Ricci for a Lagrangian
submanifold of $\widetilde{M}^n(4c)$ are given by
\begin{align}
\langle R(X,Y)Z,W \rangle &= \langle h(Y,Z),h(X,W) \rangle - \langle h(X,Z),h(Y,W) \rangle\label{2.2} \\
&+c(\langle X,W \rangle  \langle Y,Z \rangle - \langle X,Z \rangle  \langle Y,W \rangle ),\nonumber\\
(\nabla h)(X,Y,Z)&=(\nabla h)(Y,X,Z),\label{2.3}\\
\langle R^{\perp}(X,Y)JZ,JW \rangle &= \langle [A_{JZ},A_{JW}]X,Y \rangle\label{2.4} \\
&+c( \langle X,W \rangle  \langle Y,Z \rangle - \langle X,Z \rangle  \langle Y,W \rangle ),\nonumber
\end{align}
where $X$, $Y$ $Z$ and $W$ are tangent vector fields. Note that for
a Lagrangian submanifold the equations of Gauss and Ricci are
mutually equivalent.

We refer to \cite{baro82} for the construction of the standard models of indefinite complex space forms $\mathbb CP^n_s(4c)$ when $c >0$,  $\mathbb CH^n_s(4 c)$ when $c <0$ and $\mathbb C^{n}_s$. For our purposes, it is sufficient to know that there exist pseudo-Riemannian submersions, called Hopf fibrations, given by
	$$
	\Pi: S^{2n+1}_{2s}(c) \rightarrow \C P^n_s(4 c): z \mapsto z\cdot
	\C^\star
$$
if $c >0$, and by
$$
	\Pi: H^{2n+1}_{2s+1}(c) \rightarrow \C H^n_s(4 c): z \mapsto
	z\cdot \C^\star,
$$
if $c <0$, where
\begin{align*}
S^{2n+1}_{2s}(c)
	&= \{z \in \C^{n+1}\vert b_{s,n+1}(z,z) =\tfrac{1}{c}\},\\
H^{2n+1}_{s+1}(c) &= \{z \in \C^{n+1}\vert b_{s+1,n+1}(z,z)
	=\tfrac{1}{c}\}
\end{align*}
 and $b_{s,q}$ is the standard Hermitian form with index $s$ on $\mathbb C^q$. For our convenience, we will assume that we have chosen an orthonormal  basis such that the first $s$ odd terms appear with a minus sign.

In  \cite{baro82} it is shown
that
locally any indefinite complex space form is holomorphically
isometric to either $\C^n_s$, $\C P^n_s(4 c)$, or $\C
	H^n_s(4 c)$.
Remark that, by replacing the metric $<.,.>$ by $-<.,.>$, we have that $\mathbb CH^n_s(4c)$ is holomorphically anti-isometric to $\mathbb CP^n_{n-s}(-4 c)$. For that purpose, as in our case $n=2$ and $s=1$, we only need to consider $\C^2_1$ and  $\C P^2_1(4)$.

In order to study or explicitly obtain examples of Lagrangian submanifolds, it is usually more convenient to work with horizontal submanifolds. In that aspect, we first recall some basic facts from \cite{R} which relate Lagrangian submanifolds of $\mathbb C P^n_s(4c)$  to horizontal immersions in $S^{2n+1}_{2s}(c)$. Here, a horizontal immersion $f\colon M^n_s \rightarrow	S^{2n+1}_{2s}(c)$ is an immersion which satisfies $i f(p) \perp f_*(T_pM^n_s)$ for all $p\in M^n_s$, where $i =\sqrt{-1}$.

\begin{theorem}[\cite{R}]
	Let $\Pi: S^{2n+1}_{2s}(1) \rightarrow
	\mathbb C P^n_s(4)$ be the Hopf fibration.  If $f: M^n_s
	\rightarrow S^{2n+1}_{2s}(c)$ is a horizontal immersion, then
	$F = \Pi \circ f: M^n_s \rightarrow \C P^n_s(4c)$ is a
	Lagrangian immersion.

	Conversely, let $M^n_s$ be a simply connected manifold and let $F:
	M^n_s \rightarrow \C P^n_s(4)$ be a Lagrangian
	immersion. Then there exists a 1-parameter family of horizontal
	lifts $f: M^n_s \rightarrow S^{2n+1}_{2s}(1)$ such that $F = \Pi
	\circ f$. Any two such lifts $f_1$ and $f_2$ are related by $f_1 =
	e^{i\theta} f_2$, where $\theta$ is a constant.
\end{theorem}

Remark that both immersions have the same induced metric and that
the second fundamental forms of both immersions are also closely related. For more details, see \cite{R}.

\section{Minimality and Pseudo-Isotropy}\label{sec-3}
In this section we will prove Theorem \ref{th1.1}. Let us suppose that $M^2_1$ is a Lagrangian, Lorentzian surface of a complex space form $\widetilde M$. We will assume that either $\widetilde M = \mathbb C^2_1$ or $\mathbb CP^2_1(4)$.

Let $p \in M$. We say that $\{e_1,e_2\}$ is a null frame at a point $p$ if it satisfies
\begin{equation*}
<e_i,e_j>= (1-\delta_{ij}), \qquad i,j \in \{1,2\}.
\end{equation*}

In terms of a null frame, it is then clear that a Lagrangian immersion is minimal if and only if
$$h(e_1,e_2)= 0.$$
In view of the symmetries of the second fundamental form,
this implies that there exist numbers $\lambda$ and $\mu$ such that
\begin{align*}
&h(e_1,e_1)= \lambda Je_2,\\
&h(e_2,e_2)= \mu Je_1.
\end{align*}

If we now write $v= v_1 e_1+v_2 e_2$, it follows that
\begin{align*}
h(v,v)&=\mu v_2^2 J e_1+\lambda v_1^2 Je_2,\\
<h(v,v),h(v,v)>&=2 \lambda \mu v_1^2 v_2^2 =\tfrac 12 \lambda \mu <v,v>^2,
\end{align*}
which shows that a minimal surface is indeed pseudo-isotropic.

In order to show the converse, we will use
the following lemma of \cite{CFG}:
\begin{lemma}\label{isotropic}
Let $F:M \rightarrow \widetilde M$ be an isometric pseudo-Riemannian immersion. Then the immersion is (pseudo)-isotropic
if and only if for any tangent vectors $x,y,z,w \in T_p M$, we have that
\begin{multline*}
<h(x,y),h(z,w)>+<h(y,z),h(x,w)>+ <h(z,x),h(y,w)>= \\
=\widetilde\lambda(p) \{<x,y><z,w>+<y,z><x,w>+<z,x><y,w>\}.
\end{multline*}
\end{lemma}
Note that in \cite{CFG} the above lemma was formulated only for immersions into pseudo-Euclidean spaces. However, it is clear that it remains valid for arbitrary immersions in pseudo-Riemannian spaces.

Let us assume now that $M$ is a pseudo-isotropic surface. Then it follows from the previous lemma that
\begin{romanlist}
\item $h(e_1,e_1)$ is a lightlike vector, by taking $x=y=z=w=e_1$,
\item $h(e_2,e_2)$ is a lightlike vector, by taking $x=y=z=w=e_2$,
\item $h(e_1,e_2)$ is orthogonal to $h(e_1,e_1)$, by taking $x=y=z=e_1$ and $w=e_2$,\label{(iii)}
\item $h(e_1,e_2)$ is orthogonal to $h(e_2,e_2)$, by taking $x=y=z=e_2$ and $w=e_1$.
\end{romanlist}
We now write
$$h(e_1,e_2)= v_1 Je_1 +v_2 Je_2.$$
The fact that the immersion is Lagrangian then implies that
\begin{align*}
&h(e_1,e_1)=v_2 Je_1 +v_3 Je_2,\\
&h(e_2,e_2)= v_4 Je_1 + v_1 Je_2.
\end{align*}
Let us now assume that $M$ is not minimal. Then, by interchanging $e_1$ and $e_2$ if necessary, we may assume that $v_2 \ne 0$. As $h(e_1,e_1)$ is lightlike by (i), this implies that $v_3=0$. It then follows from \ref{(iii)} that $v_2 = 0$ which is a contradiction.
This completes the proof of Theorem 1.

\section{Lightlike isotropic Lagrangian, Lorentzian surfaces} \label{sec-4}

In this section we will assume that $M^2_1$ is a Lagrangian, Lorentzian lightlike pseudo-isotropic surface of a
complex space form $\widetilde M$. We  will assume that either $\widetilde M = \mathbb C^2_1$ or $\mathbb CP^2_1(4)$. We will also assume that $M$ is not pseudo-isotropic, i.e. in view of the previous section we will assume that the immersion is not minimal.
We call a surface which is lightlike pseudo-isotropic  without minimal points a proper lightlike pseudo-isotropic surface.

We again take a null frame at a point $p$, i.e. a frame $\{e_1,e_2\}$ such that
\begin{equation*}
<e_i,e_j>= (1-\delta_{ij}), \qquad i,j \in \{1,2\}.
\end{equation*}
Note that if both $h(e_1,e_1)=h(e_2,e_2)=0$, then it follows from the fact that $M$ is Lagrangian that also $h(e_1,e_2)= 0$ and therefore that $M$ is pseudo-isotropic.

We say that $M$ is proper lightlike pseudo-isotropic of Type 1 at the point $p$ if there exists a lightlike vector $v$ such that $h(v,v)$ and $Jv$ are independent. This means,  after changing $e_1$ and $e_2$ if necessary, that we may assume that
\begin{align*}
&h(e_1,e_1)= Je_2,\\
&h(e_2,e_2)= \mu Je_1+\lambda J e_2.
\end{align*}
Since $M$ is Lagrangian, we deduce from this that
\begin{equation*}
h(e_1,e_2)= \lambda Je_1.
\end{equation*}
Since $M$ is proper, we have that $\lambda \ne 0$. Hence, since $h(e_2,e_2)$ is lightlike, we deduce that $\mu =0$ and so $h(e_2,e_2)=\lambda J e_2$.

We say that $M$ is proper lightlike pseudo-isotropic of Type 2 if, for every lightlike vector $v$, $h(v,v)$ and $Jv$ are dependent.  This means, after changing $e_1$ and $e_2$ if necessary, that we may assume that
\begin{align*}
&h(e_1,e_1)= Je_1,\\
&h(e_2,e_2)= \lambda J e_2.
\end{align*}
As $M$ is Lagrangian, we deduce from this that
\begin{equation*}
h(e_1,e_2)= \lambda Je_1+ Je_2.
\end{equation*}
We see that  $M$ is indeed not minimal and the immersion is therefore proper
lightlike pseudo-isotropic.

Note that a point which belongs to the closure of Type 2 points needs to be a minimal point automatically. Therefore, it follows from the fact that $M$ is proper that $M$  can be seen as the union of
a Type 1 lightlike proper pseudo-isotropic surface and a
Type 2 lightlike proper pseudo-isotropic surface. In both cases, it is immediately clear that the null frame can be extended to a neighborhood of the point $p$.

\section{Proper Lightlike isotropic Lagrangian, Lorentzian surfaces of Type 1}\label{sec-5}

We  take a null frame in a neighborhood of the  point $p$ as constructed in the previous section. So we have a frame $\{E_1,E_2\}$ such that
\begin{equation*}
<E_i,E_j>= (1-\delta_{ij}), \qquad i,j \in \{1,2\},
\end{equation*}
and
\begin{align*}
&h(E_1,E_1)= JE_2,\\
&h(E_2,E_2)= \lambda J E_2,\\
&h(E_1,E_2)=\lambda JE_1,
\end{align*}
where $\lambda$ is a nowhere vanishing function. We write
\begin{alignat*}{2}
&\nabla_{E_1} E_1 = \alpha E_1,&\qquad  &\nabla_{E_1} E_2 = -\alpha E_2,\\
&\nabla_{E_2} E_1 = -\beta E_1,&\qquad  &\nabla_{E_2} E_2 = \beta E_2,
\end{alignat*}
where $\alpha$ and $\beta$ are functions.

\begin{lemma} We have that $\beta=0$ and $\lambda$ satisfies the following system of differential equations:
\begin{align*}
&E_1(\lambda)=-\alpha \lambda, \\
&E_2(\lambda)=0.
\end{align*}
\end{lemma}
\begin{proof}We have that
\begin{align*}
(\nabla h)(E_2,E_1,E_1)&=\nabla^\perp_{E_2}JE_2-2 h(\nabla_{E_2}E_1,E_1)\\
&=\beta JE_2 +2 \beta JE_2= 3 \beta JE_2.
\end{align*}
On the other hand, we have that
\begin{align*}
(\nabla h)(E_1,E_2,E_1)&=\nabla^\perp_{E_1}\lambda JE_1- h(\nabla_{E_1}E_2,E_1)-h(E_2,\nabla_{E_1}E_1)\\
&=(E_1(\lambda)+\alpha \lambda) JE_1 +\alpha h(E_2,E_1)-\alpha h(E_2,E_1)\\
&=(E_1(\lambda)+\alpha \lambda) JE_1.
\end{align*}
From the Codazzi equation, we therefore obtain that $\beta = 0$ and
$E_1(\lambda)= -\alpha \lambda$. Similarly from the Codazzi equation $(\nabla h)(E_1,E_2,E_2)=(\nabla h)(E_2,E_1,E_2)$, we now deduce that $E_2(\lambda)= 0$.
\end{proof}

\begin{lemma} We have that $c=0$ and $\alpha$ satisfies
\[
E_2(\alpha)= 0.
\]
\end{lemma}
\begin{proof}
We compute $[E_1,E_2] (\lambda)$ in two different ways.
We have that
\begin{align*}
[E_1,E_2] (\lambda)&=E_1(E_2(\lambda))-E_2(E_1(\lambda))\\
&=E_2(\alpha \lambda)\\
&=E_2(\alpha)\lambda
\end{align*}
and
\begin{align*}
[E_1,E_2] (\lambda)&=(\nabla_{E_1}E_2-\nabla_{E_2} E_1)(\lambda))\\
&=-\alpha E_2(\lambda)=0.
\end{align*}
Since $\lambda \ne 0$, we deduce that $E_2(\alpha)=0$.

A direct computation then yields that
\begin{align*}
R(E_1,E_2) E_1&= -\nabla_{E_2} \nabla_{E_1}E_1 -\nabla_{\nabla_{E_1}E_2} E_1\\
&=-\nabla_{E_2} (\alpha E_1)+\alpha \nabla_{E_2} E_1\\
&=-E_2(\alpha) E_1=0.
\end{align*}
So from the Gauss equation we obtain that
\begin{align*}
0&=c E_1 +A_{h(E_1,E_2)} e_1 -A_{h(E_1,E_1)} E_2\\
&=c E_1 + \lambda A_{JE_1} E_1 -A_{JE_2} E_2\\
&=c E_1.
\end{align*}
Hence the ambient space must be flat.
\end{proof}

The previous lemma immediately implies:
\begin{theorem} There does not exist a proper lightlike isotropic Lagrangian, Lorentzian surface of Type 1 in $\mathbb CP^2_1(4)$.
\end{theorem}

Moreover, we can also show that
\begin{theorem} Let $M$ be a proper lightlike isotropic Lagrangian, Lorentzian surface of Type 1 in $\mathbb C^2_1$. Then $M$ is locally congruent with
$$(\alpha(x) \tfrac{1}{\sqrt{2}} (-i,-i)+ \beta(x)\tfrac{1}{\sqrt{2}} (1,-1))e^{iv},$$
where $\alpha'(x) \beta(x)-\alpha(x) \beta'(x) \ne 0$.
\end{theorem}
\begin{proof}
We introduce vector fields $\lambda E_1$ and $\tfrac{1}{\lambda} E_2$. We have that
\begin{align*}
[\lambda E_1,\tfrac{1}{\lambda} E_2]&=-\tfrac{E_1(\lambda)}{\lambda} E_2+[E_1,E_2]\\
&=\alpha E_2-\alpha E_2 =0.
\end{align*}
Therefore, there exist coordinates $u$ and $v$ such that
$\partial u =\lambda E_1$ and $\partial v= \tfrac{1}{\lambda} E_2$. If we denote the immersion by $f$, it follows that
\begin{align*}
f_{vv}&=if_v,\\
f_{uv}&=i f_u,\\
f_{uu}&=\lambda E_1(\lambda) E_1 + \lambda^2 \nabla_{E_1}E_1+\lambda^2 h(E_1,E_1)\\
&=-\lambda^2 \alpha + \lambda^2 \alpha + \lambda^2 i E_2\\
&=\lambda^3 i f_v,
\end{align*}
where $\lambda$ is a function depending only on $u$.
Integrating the first two equations it follows that
$$f(u,v) = A_1(u) e^{iv} +A_2,$$
where $A_1$ is a vector valued function and $A_2$ is a constant.
Of course, we may assume that $A_2$ vanishes by applying a translation of $\mathbb{C}^2_1$. The third equation then tells us
that
\[
A_1''= -\lambda^3 A_1.
\]
Note that this is precisely the expression of a curve lying in the plane spanned by $A_1(0)$ and $A_1'(0)$ parametrised in such a way
that $\vert A_1 A_1'\vert$ is constant. Given that $M$ is a Lagrangian surface we must have that $A_1$ and $A_1'$ are linearly independent (over $\mathbb C$)
and that the plane spanned by $A_1$ and $i A_1'$ is real. Therefore, the constant is non-vanishing.
Since $f_u =\lambda E_1$ and $f_v = \tfrac 1\lambda E_2$, by choosing the initial conditions we may assume that $A_1(0)=\tfrac{1}{\sqrt{2}} (-i,-i)$ and
 $A_1'(0)=\tfrac{1}{\sqrt{2}} (1,-1)$.

Conversely, if we define a surface by
$$f(x,v)=(\alpha(x) \tfrac{1}{\sqrt{2}} (-i,-i)+ \beta(x)\tfrac{1}{\sqrt{2}} (1,-1))e^{iv},$$
where $\alpha'(x) \beta(x)-\alpha(x) \beta'(x) \ne 0$, we see that just as for the Euclidean arc length of a planar curve, it is possible to construct a parameter $u$ for the curve $(\alpha,\beta)$ such that $\alpha'(u) \beta(u)-\alpha(u) \beta'(u) =1$. A straightforward computation then shows that the surface $f(u,v)$ has the desired properties.
\end{proof}

\section{Proper Lightlike isotropic Lagrangian, Lorentzian surfaces of Type 2}\label{sec-6}

We  take a null frame in a neighborhood of the  point $p$ as constructed in the previous section. Then we have a frame $\{E_1,E_2\}$ such that
\begin{equation*}
<E_i,E_j>= (1-\delta_{ij}), \qquad i,j \in \{1,2\},
\end{equation*}
and
\begin{align*}
&h(E_1,E_1)= JE_1,\\
&h(E_1,E_2)= \lambda J E_1+ JE_2,\\
&h(E_2,E_2)=\lambda JE_2,
\end{align*}
where $\lambda$ is a function on $M$. We
write
\begin{alignat*}{2}
&\nabla_{E_1} E_1 = \alpha E_1,&\qquad  &\nabla_{E_1} E_2 = -\alpha E_2,\\
&\nabla_{E_2} E_1 = -\beta E_1,&\qquad  &\nabla_{E_2} E_2 = \beta E_2,
\end{alignat*}
where $\alpha$ and $\beta$ are functions.

\begin{lemma} We have that $\alpha=0$ and $\lambda$ satisfies the following system of differential equations:
\begin{align*}
&E_1(\lambda)=\beta, \\
&E_2(\lambda)=\lambda \beta.
\end{align*}
\end{lemma}
\begin{proof}We have that
\begin{align*}
(\nabla h)(E_2,E_1,E_1)&=\nabla^\perp_{E_2}JE_1-2 h(\nabla_{E_2}E_1,E_1)\\
&=-\beta JE_1 +2 \beta JE_1= \beta JE_1.
\end{align*}
On the other hand, we have that
\begin{align*}
(\nabla h)(E_1,E_2,E_1)&=\nabla^\perp_{E_1}(\lambda JE_1+JE_2)- h(\nabla_{E_1}E_2,E_1)-h(E_2,\nabla_{E_1}E_1\\
&=(E_1(\lambda)+\alpha \lambda) JE_1 -\alpha JE_2+\alpha h(E_2,E_1)-\alpha h(E_2,E_1)\\
&=(E_1(\lambda)+\alpha \lambda) JE_1-\alpha  JE_2.
\end{align*}
From the Codazzi equation, we therefore obtain that $\alpha = 0$ and
$E_1(\lambda)= \beta$. Similarly from the Codazzi equation $(\nabla h)(E_1,E_2,E_2)=(\nabla h)(E_2,E_1,E_2)$, we now deduce that $E_2(\lambda)= \lambda \beta$.
\end{proof}

\begin{lemma} We have that $\beta$ satisfies
\begin{align*}
&E_1(\beta)=-c - \lambda,\\
&E_2(\beta)=\lambda(-c-\lambda).
\end{align*}
\end{lemma}
\begin{proof}
A direct computation yields that
\begin{align*}
R(E_1,E_2) E_1&= \nabla_{E_1} \nabla_{E_2}E_1 +\nabla_{\nabla_{E_2}E_1} E_1\\
&=-\nabla_{E_1} (\beta E_1)-\beta \nabla_{E_1} E_1\\
&=-E_1(\beta) E_1.
\end{align*}
So from the Gauss equation we obtain that
\begin{align*}
-E_1(\beta)E_1&=c E_1 +A_{h(E_1,E_2)} E_1 -A_{h(E_1,E_1)} E_2\\
&=c E_1 + \lambda A_{JE_1} E_1+A_{JE_2} E_1 -A_{JE_1} E_2\\
&=c E_1+ \lambda E_1,
\end{align*}
which reduces to $E_1(\beta) =-(c+\lambda)$. In order to obtain the
$E_2$ derivative of $\beta$,
we compute $[E_1,E_2] (\lambda)$ in two different ways.
We have that
\begin{align*}
[E_1,E_2] (\lambda)&=E_1(E_2(\lambda))-E_2(E_1(\lambda))\\
&=E_1(\lambda \beta)-E_2(\beta)\\
&=\beta^2 + \lambda(-c -\lambda) -E_2(\beta)
\end{align*}
and
\begin{align*}
[E_1,E_2] (\lambda)&=(\nabla_{E_1}E_2-\nabla_{E_2} E_1)(\lambda))\\
&=\beta E_1(\lambda)=\beta^2,
\end{align*}
which clearly concludes the proof.
\end{proof}

It follows by a direct computation that
\begin{corollary}
There exists a constant $r$ such that
$$(\lambda+c)^2+ \beta^2= r^2.$$
\end{corollary}

\begin{lemma} There exist local coordinates $u$ and $v$ such that
\begin{align*}
&\tfrac{\partial}{\partial u} = E_1,\\
&\tfrac{\partial}{\partial v} = E_2-\lambda E_1.
\end{align*}
\end{lemma}
\begin{proof}
We define vector fields
\begin{align*}
&U = E_1,\\
&V = E_2-\lambda E_1,
\end{align*}
and compute
\begin{align*}
[U,V]&=[E_1, E_2-\lambda E_1]\\
&=[E_1, E_2]-E_1(\lambda) E_1\\
&=\beta E_1-E_1(\lambda) E_1=0,
\end{align*}
which proves the result.
\end{proof}

It then follows immediately from the previous systems of differential equations that $\beta$ and $\lambda$ do not depend on the variable $v$ and are determined by
\begin{align*}
&\tfrac{\partial \lambda}{\partial u} = \beta,\\
&\tfrac{\partial \beta}{\partial u} = -(c+\lambda).
\end{align*}
Therefore, note that, after a translation of the $u$ coordinate if necessary,
we may suppose that
\begin{align*}
&\lambda=-c + r \sin u,\\
&\beta= r \cos u.
\end{align*}
In the above equations, the constant $r$ is allowed to be zero. In that case, we get the special solution $\beta=0$ and $\lambda=-c$.

In view of the dimension, if necessary changing the sign of the  metric on the ambient space, we only have to consider the cases $c=0$ or $c=1$.

\subsection{Case $\mathbf{c=0}$: Lightlike isotropic Lagrangian, Lorentzian surfaces in $\mathbb C_1^2$}
We denote the immersion by $f$. It follows from the previous equations that $f$ is determined by the system of differential equations:
\begin{align*}
f_{uu}&=i f_u,\\
f_{uv}&= -re^{-iu} f_u +i f_v,\\
f_{vv}&= r e^{-iu} f_v + i r^2 (e^{-2i u} -1) f_u.
\end{align*}
It follows from the first equation  that there exist
vector valued functions $g_1$ and $g_2$ such that
$$f(u,v)=g_1(v) e^{iu} +g_2(v).$$
Substituting this into the second equation gives
$$g_2'(v)=r g_1(v),$$
and the final equation now reduces to
\begin{equation*}
g_1''(v) e^{iu} + g_2''(v)=re^{-iu}( g_1'(v) e^{iu} + g_2'(v))-r^2 (e^{-2i u} -1)e^{iu} g_1(v).
\end{equation*}
Looking at the different powers of $e^{iu}$, we deduce that
\begin{align*}
&g_1''(v)=r^2 g_1(v),\\
&g_2''(v)=rg_1'(v),\\
&0=r g_2'(v)-r^2 g_1(v).
\end{align*}
So the remaining equations are
\begin{align*}
&g_2'(v)=r g_1(v),\\
&g_1''(v)=r^2 g_1(v).
\end{align*}
The solution of the above system depends on the value of $r$.

\subsubsection{Case $r=0$}
If $r=0$ we have that $g_2(v)$ is a constant vector. Hence by applying a translation we may assume that this vector vanishes. Therefore, we have that
$$f(u,v)= (v A_1 +A_2) e^{iu},$$
for some constant vectors $A_1$ and $A_2$. We take an initial point $p =(0,0)$.
Since $\lambda(0,0)= 0$, it follows that
\begin{align*}
&E_1(0,0)= \tfrac{\partial f}{\partial u}(0,0)=iA_2,\\
&E_2(0,0)=\tfrac{\partial f}{\partial v}(0,0)=A_1.
\end{align*}
It then follows from the choice of $E_1$ and $E_2$, together with the Lagrangian condition, that $\tfrac 1{\sqrt{2}} (A_1-i A_2)$, $\tfrac i{\sqrt{2}} (A_1-i A_2)$, $\tfrac 1{\sqrt{2}} (A_1+i A_2)$, $\tfrac i{\sqrt{2}} (A_1+i A_2)$ can be identified with $(1,0)$, $(i,0)$, $(0,1)$, $(0,i)$. This implies that
\begin{align*}
&A_1=(\tfrac {\sqrt{2}}2,\tfrac {\sqrt{2}}2),\\
&A_2=(\tfrac {\sqrt{2}}2i,-\tfrac {\sqrt{2}}2i).
\end{align*}

\subsubsection{Case $r \ne 0$} In this case we have that
$$g_1(v)= A_1 e^{rv} +A_2 e^{-rv}.$$
Therefore
$$g_2'(v)=r A_1 e^{rv} +r A_2 e^{-rv},$$
which implies that, after applying a suitable translation, we have that
$$g_2(v)=A_1 e^{rv} -A_2 e^{-rv}.$$
So we find that
$$f(u,v)=(A_1 e^{rv} +A_2 e^{-rv}) e^{iu} +(A_1 e^{rv} -A_2 e^{-rv}),$$
for some constant vectors $A_1$ and $A_2$. We take again as initial point $p =(0,0)$.
Since $\lambda(0,0)= 0$, it follows that
\begin{align*}
&E_1(0,0)= \tfrac{\partial f}{\partial u}(0,0)=i(A_1+A_2),\\
&E_2(0,0)=\tfrac{\partial f}{\partial v}(0,0)=2r A_1.
\end{align*}
Or equivalently
\begin{align*}
&A_1= \tfrac{1}{2 r} E_2,\\
&A_2=-i E_1-\tfrac{1}{2 r} E_2.
\end{align*}
It then follows from the choice of $E_1$ and $E_2$, together with the Lagrangian condition, that we may assume that $E_1=(\tfrac 1{\sqrt{2}},\tfrac 1{\sqrt{2}})$ and $E_2=(-\tfrac 1{\sqrt{2}},\tfrac 1{\sqrt{2}})$, which implies that
\begin{align*}
&A_1=(-\tfrac {\sqrt{2}}{4r},\tfrac {\sqrt{2}}{4r} ),\\
&A_2=(\tfrac {(1-2ir)}{2\sqrt{2}r},\tfrac {(-1-2ir)}{\sqrt{2}2r)}).
\end{align*}

\subsubsection{Summary} Combining the previous results, we get
\begin{theorem} Let $M$ be a proper lightlike isotropic Lagrangian, Lorentzian  surface of Type 2 in $\mathbb C^2_1$. Then $M$ is congruent with one of the following surfaces:
\begin{enumerate}
\item the surface
$$f(u,v)= (v A_1 +A_2) e^{iu},$$
where $A_1=(\tfrac {\sqrt{2}}2,\tfrac {\sqrt{2}}2)$ and $A_2=(-\tfrac {\sqrt{2}}2i,\tfrac {\sqrt{2}}2i)$,
\item the surface
$$f(u,v)=(A_1 e^{rv} +A_2 e^{-rv}) e^{iu} +(A_1 e^{rv} -A_2 e^{-rv}),$$
where $r$ is a positive constant, and $A_1=(-\tfrac {\sqrt{2}}{4r},\tfrac {\sqrt{2}}{4r})$ and $A_2=(\tfrac {(1-2ir)}{2\sqrt{2}r},\tfrac {(-1-2ir)}{\sqrt{2}2r)})$.
\end{enumerate}
\end{theorem}

\subsection{Case $\mathbf{c=1}$: Lightlike isotropic Lagrangian, Lorentzian surfaces in $\mathbb CP_1^2(4)$}
We denote the horizontal lift of the immersion into $S^5_2(1)$ by $f$. It follows from the previous equations that $f$ is determined by the system of differential equations:
\begin{align*}
f_{uu}&=i f_u,\\
f_{uv}&= -(i+re^{-iu}) f_u +i f_v-f,\\
f_{vv}&= (i+r e^{-iu})(-2 f_u+f_v +2 f_u r \sin(u))-2(1-r \sin(u)) f.
\end{align*}

It follows from the first equation that there exists
vector valued functions $a_1$ and $a_2$ such that
$$f(u,v)=a_1(v) e^{iu} +a_2(v).$$
Substituting this into the second equation gives
$$a_2'(v)=-i a_2(v)+ r a_1(v).$$
The final equation now reduces to
\begin{equation*}
a_1''(v) =a_1(v) r^2 + i(a_1'(v)-a_2(v) r).
\end{equation*}
The solution of this differential equation depends on the value of $r$.

\subsubsection{Case $0 \le r <1$}
In this case we can write $r = \cos(t)$, where $t\in \, ]0,\tfrac{\pi}{2}]$. We find
that
\begin{align}
a_1(v)=& -\csc ^2(t) \left(-c_3 \sin (t) \sin (v \sin (t)) \right. \nonumber\\
      &+\left(c_1 \cos ^2(t)+i c_3\right) \cos (v \sin (t)) \nonumber\\
  &  \left. +2 i c_2 \cos (t) \sin ^2\left(\tfrac{1}{2}
   v \sin (t)\right)-c_1-i c_3\right)\label{case11}\\
a_2(v)=&-\csc ^2(t) \left(\cos (t) \left(-c_1 \sin (t) \sin (v \sin (t)) \right.\right. \nonumber\\
   &\left. +\left(c_3-i c_1\right) (\cos (v \sin (t))-1)\right)\nonumber\\
   & \left. -c_2 (\cos (v \sin (t))-i \sin (t) \sin (v \sin (t)))+c_2 \cos ^2(t)\right)
   \end{align}
We take again as initial point $p=(0,0)$. We have that $\lambda(0,0)=-1$.
It follows that
\begin{align*}
&f(0,0)=c_1+c_2,\\
&E_1(0,0)= \tfrac{\partial f}{\partial u}(0,0)= i c_1,\\
&E_2(0,0)= \tfrac{\partial f}{\partial v}(0,0)+\lambda  \tfrac{\partial f}{\partial u}(0,0)= -i (c_1+c_2) +c_3 +c_1 \cos t.
\end{align*}
So if we pick the initial conditions  $f(0,0)= (0,0,1)$, $E_1(0,0)=(\tfrac 1{\sqrt{2}},\tfrac 1{\sqrt{2}},0)$ and $E_2(0,0)=(-\tfrac 1{\sqrt{2}},\tfrac 1{\sqrt{2}},0)$,  we find that
\begin{align}
c_1&=\left(-\frac{i}{\sqrt{2}},-\frac{i}{\sqrt{2}},0\right),\\
c_2&=\left(\frac{i}{\sqrt{2}},\frac{i}{\sqrt{2}},1\right),\\
c_3&=\left(i \left(\frac{\cos (t)}{\sqrt{2}}+\frac{i}{\sqrt{2}}\right),i
   \left(\frac{\cos (t)}{\sqrt{2}}-\frac{i}{\sqrt{2}}\right),i\right). \label{case12}
\end{align}

\subsubsection{Case $r=1$}We obtain as solution of the differential equation that
\begin{align}
&a_1(v)=\frac{1}{2} \left(c_1 \left(v^2+2\right)+v \left(2 c_3-i
   \left(c_2-c_3\right) v\right)\right), \label{case21}\\
     &a_2(v)=\frac{1}{2} \left(2 c_2+v \left(c_3 v+\left(c_1-i c_2\right) (2-i
   v)\right)\right).
   \end{align}
We take again as initial point $p=(0,0)$. We have that $\lambda(0,0)=-1$.
It follows that
\begin{align*}
&f(0,0)=c_1+c_2,\\
&E_1(0,0)= \tfrac{\partial f}{\partial u}(0,0)= i c_1,\\
&E_2(0,0)= (1-i) c_1-i c_2 +c_3.
\end{align*}
So, if we pick the initial conditions  $f(0,0)= (0,0,1)$, $E_1(0,0)=(\tfrac 1{\sqrt{2}},\tfrac 1{\sqrt{2}},0)$ and $E_2=(-\tfrac 1{\sqrt{2}},\tfrac 1{\sqrt{2}},0)$,  we find that
\begin{align}
c_1&=\left(-\frac{i}{\sqrt{2}},-\frac{i}{\sqrt{2}},0\right),\\
c_2&=\left(\frac{i}{\sqrt{2}},\frac{i}{\sqrt{2}},1\right),\\
c_3&=\left(\frac{i-1}{\sqrt{2}},\frac{1+i)}{\sqrt{2}},i\right).\label{case22}
\end{align}

\subsubsection{Case $r >1$}
We obtain as solution of the differential equation that
\tiny
\begin{align}
&a_1(v)=\frac{c_3 \sqrt{r^2-1} \sinh \left(\sqrt{r^2-1} v\right)+\left(c_1 r^2-i
   c_2 r+i c_3\right) \cosh \left(\sqrt{r^2-1} v\right)+i c_2 r-c_1-i
   c_3}{r^2-1},\label{case31}\\
     &a_2(v)=\frac{\sqrt{r^2-1} \left(c_1 r-i c_2\right) \sinh \left(\sqrt{r^2-1}
   v\right)+\left(-c_2+\left(c_3-i c_1\right) r\right) \cosh
   \left(\sqrt{r^2-1} v\right)+r \left(c_2 r+i c_1-c_3\right)}{r^2-1}.
   \end{align} \normalsize
We take again as initial point $p=(0,0)$. We have that $\lambda(0,0)=-1$.
It follows that
\begin{align*}
&f(0,0)=c_1+c_2,\\
&E_1(0,0)= \tfrac{\partial f}{\partial u}(0,0)= i c_1,\\
&E_2(0,0)= -i (c_1+c_2) +c_3 +c_1 r.
\end{align*}
So, if we pick the initial conditions $f(0,0)= (0,0,1)$, $E_1(0,0)=(\tfrac 1{\sqrt{2}},\tfrac 1{\sqrt{2}},0)$ and $E_2=(-\tfrac 1{\sqrt{2}},\tfrac 1{\sqrt{2}},0)$,  we find that
\begin{align}
c_1&=\left(-\frac{i}{\sqrt{2}},-\frac{i}{\sqrt{2}},0\right),\\
c_2&=\left(\frac{i}{\sqrt{2}},\frac{i}{\sqrt{2}},1\right),\\
c_3&=\left(\frac{i(i+r)}{\sqrt{2}},\frac{(1+ir)}{\sqrt{2}},i\right).\label{case32}
\end{align}

\subsubsection{Summary} Combining the previous results we get the following theorem which finishes our classification.

\begin{theorem} Let $M$ be a proper lightlike isotropic Lagrangian, Lorentzian surface in $\mathbb CP^2_1(4)$. Then the Hopf lift of $M$ is congruent with  one of the following immersions into $S^5_1(1)$ given by $f(u,v)=a_1(v) e^{iu} +a_2(v)$, where either
\begin{enumerate}
\item $a_1,a_2,c_1,c_2,c_3$ are as described in \eqref{case11}--\eqref{case12}, or
\item $a_1,a_2,c_1,c_2,c_3$ are as described in \eqref{case21}--\eqref{case22}, or
\item $a_1,a_2,c_1,c_2,c_3$ are as described in \eqref{case31}--\eqref{case32}.
\end{enumerate}
\end{theorem}

\end{document}